\documentclass{amsart}

\usepackage{amssymb}

\usepackage{hyperref}

\newtheorem{theorem}{Theorem}[section]
\newtheorem{lemma}[theorem]{Lemma}

\theoremstyle{definition}

\newtheorem{proposition}[theorem]{Proposition}

\theoremstyle{remark}
\newtheorem{remark}[theorem]{Remark}

\numberwithin{equation}{section}

\newcommand{\rn}[1]{\mathbb{R}^{#1}}

\newcommand{\zero}{{\bf 0}}

\newcommand{\bq}{\begin{equation}}

\newcommand{\eq}{\end{equation}}

\begin{document}

\title[Equivalence Between Reversible and Hamiltonian Dynamical Systems]
{Formal Equivalence Between Normal Forms of Reversible and Hamiltonian Dynamical Systems}

\author[R.M. Martins]{Ricardo Miranda Martins}
\address{Department of Mathematics, IMECC, Unicamp, CEP 13083-970, Campinas SP, Brazil}
\curraddr{}
\email[R.M. Martins]{rmiranda@ime.unicamp.br}

\thanks{$^*$ The author is supported by a FAPESP--BRAZIL grant 2007/05215-4. }

\subjclass[2010]{34C14, 34C20, 37J15, 37J40}

\begin{abstract}
We show the existence of formal equivalences between reversible and Hamiltonian vector fields. The main tool we employ is the normal form theory. 
\end{abstract}

\maketitle

\section{Introduction}

In this paper we deal with $C^\infty$ reversible vector fields on
 $\mathbb{R}^{2n}$. These objects are assumed to have a simple and symmetric equilibrium at ${\bf 0}.$ Our main concern is to investigate whether there exists a sequence of changes of coordinates and time-reparametrizations around the origin that turn
 the truncations of the Taylor expansion of the vector field at any given order into a Hamiltonian vector field.

One of the main motivations of this work is a question due to Arnold in his book {\it Arnold's Problems} (\cite{ap}):

{\it Question 1984-23: Develop a ``supertheory'' whose even component corresponds to reversible systems, and odd one to Hamiltonian systems.}

Several years before Arnold posing this question, the similarity between reversible and Hamiltonian vector fields was investigated by some authors. As far we know, the first to consider this question was Birkhoff \cite{birkhoff1} at the beginning of the last century. Price \cite{baley1} and Devaney \cite{devaney} also noted such similarity.

A lot of Hamiltonian results has been successfully adapted to reversible systems, for instance KAM Theory and the Lyapunov Center Theorems. Since then, authors try to understand the real role of the symplectic structure on the Hamiltonian-oriented results, and when it is possible to replace the symplectic structure by a time-reversal symmetry. 

Our concern is to unify both theories from a formal point of view. We say that two vector fields are formally conjugate if there exists  a formal change of coordinates transforming one vector field to the other (to any given order, without regard for the convergence of the transformation in the limit where the order goes to infinity).  Two vector fields $X,Y:\rn{n}\rightarrow\rn{n}$  are said to be formally orbitally equivalent if there is a smooth function $f:\mathbb{R}^n\to\mathbb{R}$ with no zeros near $0$, so that $f\cdot X$ ($X$ multiplied by $f$)  is formally conjugate to $Y$. The multiplication by $f$ has the interpretation of a time-reparametrization of the orbits of $X$. Formal conjugacy between two vector fields implies formal orbital equivalence but not vice versa.

We consider the question of whether normal forms of reversible vector fields can be made Hamiltonian using change of coordinates and time-reparametrizations. Low dimensional versions of this problem were discussed in \cite{mt-cpaa,sand,rp}. Here we focus on $2n$-dimensional reversible vector fields with a elliptic or saddle singularity in the origin. 

Let $\Omega_0^{2n}$ be the set of germs of $C^\infty$ vector fields $X:\rn{2n},0\rightarrow\rn{2n},0$ with $X({\bf 0})={\bf 0}$ and such that ${\bf 0}$ is a simple singularity, that is, $\det DX({\bf 0})\neq 0$.

Fix an involution $\varphi:\rn{2n},0\rightarrow\rn{2n},0$ with $\varphi({\bf 0})={\bf 0}$ and $\dim {\rm Fix}(\varphi)=n$. Recall that a vector field $X:\rn{2n}\rightarrow\rn{2n}$ is $\varphi$-reversible if $\varphi\circ X_t=X_{-t}\circ\varphi$, where $X_t$ is the flow of $X$. Denote by $\Omega_0^{2n}(\varphi)\subset \Omega_0^{2n}$ the subset of $\varphi$-reversible vector fields.

We distinguish two subsets of $\Omega_0^{2n}(\varphi)$: $\Omega_{0,e}^{2n}(\varphi)$ and $\Omega_{0,s}^{2n}(\varphi)$. If $X\in\Omega_{0,e}^{2n}(\varphi)$ then $\zero$ is an elliptic singularity for $X$, and if $X\in\Omega_{0,s}^{2n}(\varphi)$, then $\zero$ is a saddle singularity for $X$.

For $X\in\Omega_{0,e}^{2n}(\varphi)$, we denote $\sigma(DX(\zero))=\{\pm \alpha_1 i,\ldots,\pm \alpha_n i\}$, and for $X\in\Omega_{0,s}^{2n}(\varphi)$ we denote $\sigma(DX(\zero))=\{\pm \beta_1,\ldots,\pm \beta_n\}$, for $\alpha_j,\beta_j\geq 0$, $j=1,\ldots, n$.

For $X\in\Omega_{0,e}^{2n}(\varphi)$, the origin is said to be elliptic $p_1:\ldots:p_k$-resonant if there is $\lambda\neq 0$ such that $\alpha_{j}=\lambda p_j$, for $j=1,\ldots, k\leq n$ and $\{\alpha_{k+1},\ldots,\alpha_n\}$ is linearly independent over $\mathbb Q$ (possibly after a reordering). If $\{\alpha_1,\ldots,\alpha_n\}$ is a linear independent set over $\mathbb Q$, then the origin is said elliptic nonresonant. A similar definition holds for saddle singularities.

Denote the subsets of vector fields with a nonresonant singularity at $\zero$ of elliptic and saddle types, respectively, by $\Omega_{0,e,\infty}^{2n}(\varphi)\subset \Omega_{0,e}^{2n}(\varphi)$ and $\Omega_{0,s,\infty}^{2n}(\varphi)\subset \Omega_{0,s}^{2n}(\varphi)$.


Our main results are the following.\\

\noindent {\bf Theorem A:} Let $X\in\Omega_{0,e,\infty}^{2n}(\varphi) \, \cup \, \Omega_{0,s,\infty}^{2n}(\varphi)$ and suppose that $j^3 X$ is a Hamiltonian vector field. Then $X$ is generically formally orbitally equivalent to a Hamiltonian vector field.\\

In dimensions 2 and 4, the resulting Hamiltonian vector field in Theorem A can be taken decoupled and polynomial, as proved in \cite{rp,mt-cpaa}. For $n\geq 3$ (dimension greater than 6) it is not possible to attain such result, and generically the simplest Hamiltonian vector field that can be obtained is a decoupled vector field, up to some lower order jets.\\

\noindent {\bf Corollary 1:} Let $X\in\Omega_{0,e,\infty}^{6}(\varphi) \, \cup \, \Omega_{0,s,\infty}^{6}(\varphi)$ and suppose that $j^3 X$ is a Hamiltonian vector field. Then $X$ is generically formally orbitally equivalent to a Hamiltonian vector field $Y$ such that $Y-j^3Y$ is a decoupled Hamiltonian vector field.\\

In dimension 6 it is possible to attain the Hamiltonian structure without using time-reparametrizations.\\

\noindent {\bf Corollary 2:} Let $X\in\Omega_{0,e,\infty}^{6} (\varphi) \, \cup \, \Omega_{0,s,\infty}^{6}(\varphi)$ and suppose that $j^3 X$ is a Hamiltonian vector field. Then $X$ is generically conjugated to a Hamiltonian vector field.\\

\begin{remark}The hypothesis ``{\it $j^3 X$ is a Hamiltonian vector field}" in Theorem A and Corollaries 1 and 2 is general in the following sense: Let $Y\in\Omega_{0,e,\infty}^{6}(\varphi) \, \cup \, \Omega_{0,s,\infty}^{6} (\varphi)$ be a reversible vector field. If $Y$ is conjugated to a vector field $X$ such that Theorem A or Corollaries 1 or 2 holds for $X$, then $j^3 Y$ is hamiltonian.
\end{remark}

While the above results are for reversible vector fields with nonresonant equilibria, using some normal form results in \cite{mt-aabc} it is possible to adapt them to reversible-equivariant vector fields with a resonant equilibria. 

Given a finite group of involutive diffeomorphisms $G$ and a group homomorphism $\rho:G\rightarrow\{-1,1\}$, we say that a vector field $X$ is $G$-reversible-equivariant if $\psi\circ X_t= X_{\rho(\psi)} \circ \psi$ for all $\psi\in G$. Details on reversible-equivariant vector fields can be seen on \cite{paty}.

Here we consider $G=\langle g_1,g_2\rangle$ a group generated by two involutions and $G$-reversible-equivariant vector fields that are $g_1$-reversible and $g_2$-reversible. For this kind of reversible-equivariant vector fields, simplified normal forms are given in \cite{mt-aabc}. We fix $G=\mathbb D_4$.

Consider $\Omega_{0,e}^{2n}(\mathbb D_4) \, \cup \, \Omega_{0,e}^{2n}$ the subset of $\mathbb D_4$-reversible vector fields.

The main results for $\mathbb D_4$-reversible-equivariant vector fields are the following.\\

\noindent {\bf Theorem B:} Let $X\in\Omega_{0,e}^{4}(\mathbb D_4)$. Consider that the origin is $r_1:r_2$-resonant. Suppose that $r_1,r_2$ are odd numbers with $r_1r_2>1$. Then $X$ is generically formally orbitally equivalent to a Hamiltonian vector field.\\

Theorem B generalizes the conjugacy results in \cite{mt-cpaa}, but restricted to reversible-equivariant vector fields.\\

This paper is organized as follows. In section 2 we present some results on reversible and Hamiltonian vector fields and normal forms. In section 3 we present the state of the art of the equivalence problem between reversible and Hamiltonian vector fields. In section 4 we provide a detailed statement of Theorem A and in section 5 we prove Theorem A and its corollaries. In section 6 we prove Theorem B.

\section{Background}

In this section we present some basic results that we shall need throughout this paper. Classical references on these subjects are \cite{marco1} and \cite{mar1}.

\subsection{Reversible vector fields}

A vector field $X:\rn{2n}\rightarrow\rn{2n}$ is said to be $\varphi$-reversible if $\varphi_*X=-X$, where $\varphi:\rn{2n}\rightarrow\rn{2n}$ is an involutive diffeomorphism. This implies that if ${\bf x}(t)$ is a solution of \[\dot {\bf x}=X({\bf x})\] then $\varphi({\bf x}(-t))$ is also a solution. In this paper we assume that $\dim {\rm Fix}(\varphi)=\frac n2$. 
We have the following classical result.

\begin{lemma}[Montgomery-Bochner Theorem]\label{tmb}Let $\varphi:\rn{2n}\rightarrow\rn{2n}$ be a diffeomorphism with $\varphi^2=Id$ and $\varphi({\bf 0})={\bf 0}$. Then $\varphi$ and $D\varphi(0)$ are conjugate.
\end{lemma}
\begin{proof}The proof is straightforward. If $\phi=Id+D\varphi(0)\varphi$, then:
\begin{eqnarray*}\phi(\varphi({\bf x}))&=&(Id+D\varphi(0)\varphi)(\varphi({\bf x}))\\
&=&\varphi({\bf x})+(D\varphi(0)\varphi)(\varphi({\bf x}))\\
&=&\varphi({\bf x})+D\varphi(0){\bf x}\\
&=&D\varphi(0)({\bf x}+D\varphi(0)\varphi({\bf x}))\\
&=&D\varphi(0)\phi({\bf x}).
\end{eqnarray*}
\end{proof}

Let $X:\rn{2n}\rightarrow\rn{2n}$ be a germ of $\varphi$-reversible vector field and suppose $X({\bf 0})=\varphi({\bf 0})={\bf 0}$ and $\dim {\rm Fix}(\varphi)=n$. By Lemma \ref{tmb}, we can choose a linear $\varphi:\rn{4}\rightarrow\rn{4}$ to work with. Fix for instance $\varphi_0(x_1,y_1,\ldots,x_n,y_n)=(x_1,-y_1,\ldots,x_n,-y_n)$.

\subsection{Hamiltonian vector fields}

We present a definition of Hamiltonian vector field adequate for a local study. It can be derived from the most general definitions by the Darboux theorem (see \cite{mar1}).

Let
\[\Omega=\left(
\begin{array}{cc}
0&Id_n\\
-Id_n&0
\end{array}\right),\]
where $Id_n$ is the $n\times n$ identity matrix. A matrix $M$ is said to be sympletic if $M^t\Omega M=\Omega$.

A vector field $X:\rn{2n}\rightarrow\rn{2n}$ is said to be Hamiltonian if there exists a function $H:\rn{2n}\rightarrow\rn{}$ and a sympletic matrix $J$ such that $X=J\nabla H$, where $\nabla H$ is the gradient of $H$. In this paper we fix the sympletic matrix \[J=\left(
\begin{array}{ccccc}
0&-1&&&\\
1&0&&&\\
&&\ddots&&\\
&&&0&-1\\
&&&1&0
\end{array}\right)\]

\subsection{Normal forms\label{sec.fn}}
Consider a system of differential equations
\bq\label{fn1}\dot {\bf x}=f({\bf x}),\eq
with ${\bf x}\in\rn{n}$ and $f({\bf 0})={\bf 0}$.
We seek a coordinate change under which system \eqref{fn1} becomes ``simpler'' near the origin, that is, we want to eliminate some terms of the Taylor expansion of $f$ around ${\bf 0}$.

Write $f({\bf x})=A{\bf x}+f_2({\bf x})+f_3({\bf x})+h.o.t.$, where $A=Df({\bf 0})$ and $f_l$ is homogeneous of degree $l$. Consider a change of coordinates ${\bf x}={\bf y}+h_2({\bf x})$ with $h_2$ homogeneous of degree 2. With this change of coordinates, and using the fact that $(Id+Dh_2)^{-1}=Id-Dh_2+h.o.t.$, \eqref{fn1} writes as
\[{\bf y}=A{\bf y}+Ah_2({\bf y})-Dh_2({\bf y})A{\bf y}+f_2({\bf y})+h.o.t.\]

Denote $L_A^{(2)}(h_2)({\bf y})=Ah_2({\bf y})-Dh_2({\bf y})A{\bf y}$. If $f_2$ is in the range of $L_A^{(2)}$, then we can choose $h_2$ that eliminates the quadratic terms in \eqref{fn1}. However, in general $L_A^{(2)}$ is not surjective. The terms of $f_2$ that do not lie in the range of $L_A^{(2)}$ are called resonant terms. The resonant terms cannot be removed from \eqref{fn1}.

It should be clear that the procedure above can be repeated to higher order terms using the operator $L_A^{(m)}$, thus eliminating all the nonresonant terms of \eqref{fn1}.

\begin{theorem}[Poincar\'e-Dulac, \cite{pd1}]\label{tpd}
A system of differential equations
\bq\label{fg.1}\dot {\bf x}=A{\bf x}+f_2({\bf x})+h.o.t.\eq
can be formally reduced to
\bq\label{fg.2}\dot {\bf y}=A{\bf y}+\sum_{k=2}^\infty g_k({\bf y})\eq
where $g_k$ is formed by resonant terms, with $D g_k({\bf y})A{\bf y}-Ag_k({\bf y})=0$, for all $k\geq 2$. System \eqref{fg.2} is said to be the normal form of system \eqref{fg.1}.
\end{theorem}

Note that the resonant terms in the Taylor expansion of $f$ depend only on the linear part of $f$. They are called resonant terms because they are related to the existence of resonances relations among the eigenvalues of $A=Df({\bf 0})$. Indeed, assume $A$ is a diagonal matrix with eigenvalues $\lambda_1,\ldots,\lambda_n\in\mathbb C$. Let $H^{m,n}$ denote the space of the $m$-homogeneous monomials in $\rn{n}$, ${\bf x}^m\frac{\partial}{\partial x_s}$, and $L_A^{(m)}:H^{m,n}\rightarrow H^{m,n}$ be defined by $L_A^{(m)}(h_m)({\bf x})=Ah_m({\bf x})-Dh_m({\bf x})A{\bf x}$, where ${\bf x}^m=x_1^{m_1}\cdots x_n^{m_n}$.

The eigenvectors of $L_A^{(m)}$ are the monomials ${\bf x}^m\frac{\partial}{\partial x_s}$ and the following relation holds \[L_A^{(m)}\left({\bf x}^m\frac{\partial}{\partial x_s}\right)=[(m,\lambda)-\lambda_s]{\bf x}^m\frac{\partial}{\partial x_s},\] where $(m,\lambda)=\sum_j m_j\lambda_j$ (see \cite{pd1}).

So if the the relation $(m,\lambda)-\lambda_s=0$ for some $s$ and $m$, then the monomial ${\bf x}^m\frac{\partial}{\partial x_s}$ is resonant (this monomial is in the kernel of $L_A^{(m)}$).

We remark that it is possible to compute a reversible normal form for a reversible vector field, see \cite{gaeta}.

For Hamiltonian vector fields, there is a special normal form called Birkhoff normal form. In this case, the homological operator acts directly on the Hamiltonian function, and not on the underlying vector field. See \cite{fnbirkhoff}) for more details on the Birkhoff normal forms.

\section{State of the art}

In a paper from 1917 \cite{birkhoff1}, Birkhoff was studying some conservative dynamical systems with two degree of freedom. He realized that imposing some ``reversibility'' (this term was used without any formal definition) conditions would simplify the analysis of his model, a set of equations of motion of a Lagrangian system.

Using these new conditions, he could simplify his model and deduce the existence of periodic orbits. It was clear for Birkhoff that the reversible results should be valid for this original hamiltonian context, but the formal proof for the ``irreversible'' case would be harder.

In 1935 \cite{baley1}, Price wrote a survey on reversible systems, with the emphasis on those which have an oval of zero velocity. The reversible vector fields he considered were a special case of conservative vector fields. Price constructed such vector fields in a surface of revolution, by rotating a vector field with a unique singular point. The revolution of the singularity created the oval of singular points.

The dichotomy between reversible and hamiltonian vector fields started to be explicit, as much of the results of Price do not really depended on the conservative structure, but on the symmetry conditions.

The fundamental property of $R$-reversible vector fields, that the orbit through $R(x_0)$ is just the orbit through $x_0$ reflected by $R$ and traced in the backwards, is stated in the paper of Price. The involution $R$ was always taken as $R(x,y)=(x,-y)$. As in the Birkhoff's paper, no formal definition was given by Price.

Just in 1976, Devaney \cite{devaney} gave definition of reversible vector fields, independently of the conservative conditions. Devaney showed Kupka-Smale Theorem for reversible vector fields, based on the Robinson's proof of Kupka-Smale for hamiltonian vector fields.

Two classical results for hamiltonian vector fields were also adapted to the reversible context by Devaney: the Lyapunov Center Theorem, and a result on symmetry homoclinic orbits (under some conditions, they are limit of periodic orbits).

Some years ago, Moser \cite{sev} presented a version of the KAM Theory for reversible vector fields. The proof in the reversible context was much more easier than the original, for hamiltonian vector fields.

In 1983, Arnold \cite{ap} posed the question in the introduction of this paper, suggesting that every hamiltonian result could have be adapted to reversible systems, and reciprocally.

It seems that the first progress in this direction was made in 1994 by Sanders, van der Meer and Vanderbauwhere \cite{sand}. They have shown that every reversible non-Hamiltonian vector field in $\rn{4}$ at nonsemisimple $1:1$-resonance can be split into a Hamiltonian and a non-Hamiltonian part, and the non-Hamiltonian part can be eliminated using normal form reductions. The use of normal form reductions obligated the results to be formal.

In \cite{mt-cpaa}, Martins and Teixeira solved the problem of the similarity of reversible and hamiltonian vector fields for some cases. They proved that reversible vector fields in $\rn{4}$ with an simple equilibrium of elliptic, saddle or saddle-elliptic type can be made Hamiltonian using changes of coordinates. They also showed that using time-reparametrizations and changes of coordinates, it is possible to transform a reversible vector field in a polynomial hamiltonian vector field.

We review some results in \cite{mt-cpaa}. Let $\sigma(A)$ denote the set of the eigenvalues of the matrix $A$.

\begin{proposition}\label{prop2d} Let $X:\rn{2}\rightarrow\rn{2}$ be a vector field with $X({\bf 0})={\bf 0}$ and $A=DX({\bf 0})$, with $\det(A)\neq 0$. Suppose $X$ is $\varphi$-reversible, with $\varphi(x,y)=(x,-y)$. Then $X$ is formally conjugated to a Hamiltonian vector field.
\end{proposition}
\begin{proof} Note that $\sigma(A)=\{ \alpha i,-\alpha i\}$, $\alpha>0$, or $\sigma(A)=\{\beta,-\beta\}$, $\beta>0$. So the normal form $\tilde X$ of $X$ is given by
\bq\label{fn.2d.ee}
\begin{array}{lcl}
\dot x&=&-\alpha y -yf(x^2+y^2),\\
\dot y&=&\alpha x+xf(x^2+y^2)
\end{array}
\eq
in the first case and by \bq\label{fn.2d.ss}
\begin{array}{lcl}
\dot x&=&\beta x +xf(xy),\\
\dot y&=&-\beta y-yf(sy)
\end{array}
\eq
in the second case, where $f,g$ are formal functions without constant terms. Recall that $X$ and $\tilde X$ are formally conjugated.
System \eqref{fn.2d.ee} is Hamiltonian with respect to \[H_e({\bf x})=\dfrac{\alpha}2(x^2+y^2)+\int yf(x^2+y^2)\,dy\] and system \eqref{fn.2d.ss} is Hamiltonian with respect to \[H_s({\bf x})=-\beta xy-x\int g(xy)\, dy.\]
So a reversible vector field in $\rn{2}$ with a simple equilibria at ${\bf 0}$ is formally Hamiltonian around ${\bf 0}$.
\end{proof}

\begin{proposition}\label{prop.mtcpaa} Let $X:\rn{4}\rightarrow\rn{4}$ be a vector field with $X({\bf 0})={\bf 0}$ and $A=DX({\bf 0})$ with $\det(A)\neq 0$. Consider also that if $\lambda\in\sigma(A)$, then $\lambda\in\rn{}$ or $i\lambda\in\rn{}$. Suppose that $X$ is $\varphi$-reversible, with $\dim Fix(\varphi)=2$. Then $X$ is formally conjugated to a Hamiltonian vector field, and formally orbitally equivalent to a polynomial Hamiltonian vector field.
\end{proposition}

The proof of Proposition \ref{prop.mtcpaa} is much more complicated than the proof of Proposition \ref{prop2d}, as in the $4D$-case the normal forms are not directly hamiltonian.

\section{Detailed statements of the results\label{sec.detail}}

Recall that two vector fields $X$ and $Y$ are formally conjugated if there is a formal change of coordinates $\Psi$ such that $\Psi_*X=Y$. This means that for every $k<\infty$, there is an analytic change of coordinates $\Psi_k$ such that $j^k ({\Psi_k}_*X-Y)=0$.

A weak equivalence relation is the formal orbital equivalence. Given two vector fields $X,Y$ with a common singularity $x_0$, we say that $X,Y$ are formally orbitally equivalent around $x_0$ if there is a function $f$ without zeroes near $x_0$ such that $f\cdot X$ and $Y$ are formally conjugated. Note that formal conjugacy implies formal orbital equivalence (take $f=1$).

Let $X\in\Omega_{0,e,\infty}^{2n}(\varphi)$. Let $\tilde X$ be a normal form for $X$ around ${\bf 0}$. So $\tilde X$ writes as $\tilde X({\bf x})=A{\bf x}+f_3({\bf x})+f_5({\bf x})+\ldots$, where $f_k$ is $k$-homogeneous (the even order terms in the Taylor expansion of $\tilde X$ vanishes).

The nonresonant hypothesis on the equilibria ${\bf 0}$ assures that
\bq\label{f2jm1}f_{2j+1}({\bf x})=\left(
\begin{array}{l}
-y_1 f_1^{[j]}(\Delta_1,\ldots,\Delta_n)\\
x_1  f_1^{[j]}(\Delta_1,\ldots,\Delta_n)\\
\vdots\\
-y_{n} f_n^{[j]}(\Delta_1,\ldots,\Delta_n)\\
x_n  f_n^{[j]}(\Delta_1,\ldots,\Delta_n)
\end{array}
\right)\eq
where $f_i^{[j]}$ is a $j$-homogeneous polynomial and $\Delta_j=x_j^2+y_j^2$ for ${\bf x}=(x_1,y_1,\ldots,x_n,y_n)$ (note that $f_i^{[j]}$ is a $2j$-degree polynomial in $x_1,y_1,\ldots,x_n,y_n$).

Note that, for $n=2$ (dimension 4), if we consider $X$ a $\mathbb D_4$-reversible-equivariant vector field such that $\zero$ is a $r_1:r_2$-resonant vector field with $r_1,r_2$ odd numbers, $r_1r_2>1$, then $f_{2j+1}$ also has this expression. More details, see \cite{mt-aabc}.

Our aim is to take $\tilde X$ to a Hamiltonian vector field $Y$. Write $Y({\bf x})=A{\bf x}+g_3({\bf x})+g_5({\bf x})+\ldots$, where $g_k=J\nabla H_{k+1}$ for some $k+1$-homogeneous function $H_{k+1}:\rn{2n}\rightarrow\rn{}$.

For $k\geq 1$, consider a change of coordinates $\Psi_{2k+1}=Id+\psi_{2k+1}$ with $\psi_{2k+1}$ a vector valued $(2k+1)$-homogeneous function and a time-reparametrization $\rho_{2k+2}=1+\theta_{2k+2}$, for $\theta_{2k+2}$ a $2k+2$-homogeneous function.

If \bq\label{ee.1}j^{2k+3}\left( {{\Psi_{2k+1}}_*}(\rho_{2k+2}\cdot \tilde X)-Y\right)=0\eq
holds for every $k$, then $\tilde X$ is formally orbitally equivalent to $Y$. As $X$ is formally conjugated to $\tilde X$, then $X$ is formally orbitally equivalent to $Y$.

Consider $\psi_{2k+1}$ in the form
\bq\label{psi2km1}
\psi_{2k+1}=\left(
\begin{array}{l}
x_1\sigma_1^{[k]}( \Delta_1,\ldots,\Delta_n)\\
y_1\sigma_1^{[k]}( \Delta_1,\ldots,\Delta_n)\\
\vdots\\
x_n\sigma_n^{[k]}( \Delta_1,\ldots,\Delta_n)\\
y_n\sigma_n^{[k]}( \Delta_1,\ldots,\Delta_n)
\end{array}
\right)
\eq
with $\sigma_j^{[k]}$ a $k$-degree polynomial.

Suppose \eqref{ee.1} is true for every $k'< k$. For $k'=k$ equation \eqref{ee.1} is equivalent to the system
\bq\label{SYSTEM}f_{2k+3}-g_{2k+3}+A{\bf x}\cdot \theta_{2k+2}=2f_3^\Sigma,\eq
where $f_3^\Sigma=f_3(\sigma_1 \Delta_1,\ldots,\sigma_n\Delta_n)$ and $\sigma_j=\sigma_j(\Delta_1,\ldots,\Delta_n)$.

So to prove Theorem A for $X\in\Omega_{0,e,\infty}^{2n}(\varphi)$ it is suffice to show that, under some genericity condition, \eqref{SYSTEM} has a solution with respect to $g_{2k+3}$, $\theta_{2k+2}$ and $f_3^\Sigma$ for every $k\geq 1$, as the first step of the induction process ($j^3 \tilde X$ is Hamiltonian) is assumed true.

The proof of Theorem A for $X\in\Omega_{0,s,\infty}^{2n}(\varphi)$, the saddle case, is very similar to the proof of the elliptic case. We just mention the main changes.

The expression for $f_{2j+1}$, given in \eqref{f2jm1}, changes to
\bq\label{f2jm1new}f_{2j+1}({\bf x})=\left(
\begin{array}{l}
x_1 f_1^{[j]}(\Gamma_1,\ldots,\Gamma_n)\\
-y_1  f_1^{[j]}(\Gamma_1,\ldots,\Gamma_n)\\
\vdots\\
x_{n} f_n^{[j]}(\Gamma_1,\ldots,\Gamma_n)\\
-y_n  f_n^{[j]}(\Gamma_1,\ldots,\Gamma_n)
\end{array}
\right)\eq
where $f_i^{[j]}$ is a $j$-homogeneous polynomial and $\Gamma_j=x_jy_j$ for ${\bf x}=(x_1,y_1,\ldots,x_n,y_n)$.

In the saddle case we also have to change $\psi_{2k+1}$, from \eqref{psi2km1} to
\bq\label{psi2km1new}
\psi_{2k+1}=\left(
\begin{array}{l}
x_1\sigma_1^{[k]}( \Gamma_1,\ldots,\Gamma_n)\\
y_1\sigma_1^{[k]}( \Gamma_1,\ldots,\Gamma_n)\\
\vdots\\
x_n\sigma_n^{[k]}( \Gamma_1,\ldots,\Gamma_n)\\
y_n\sigma_n^{[k]}(\Gamma_1,\ldots,\Gamma_n)
\end{array}
\right)\eq
with $\sigma_j^{[k]}$ a $k$-degree polynomial.

The final equation, \eqref{SYSTEM}, remains the same, but now with $f_3^\Sigma=f_3(\sigma_1 \Gamma_1,\ldots,\sigma_n\Gamma_n)$ and $\sigma_j=\sigma_j(\Gamma_1,\ldots,\Gamma_n)$.

If $n=3$ (dimension $6$), then is possible to prove a more specific result. Let $X\in \Omega_{0,e,\infty}^6(\varphi)$ be a reversible vector field such that $j^3\tilde X$ is a Hamiltonian vector field. Then a normal form for $X$ is given by
\[\tilde X: 
\begin{array}{lcl}
\dot x_1&=&-\alpha_1 y_1-y_1 \sum_{i+j+k=1}^\infty a_{i,j,k}\Delta_1^i\Delta_2^j\Delta_3^k,\\
\dot y_1&=&\alpha_1 x_1 +x_1\sum_{i+j+k=1}^\infty a_{i,j,k}\Delta_1^i\Delta_2^j\Delta_3^k,\\
\dot x_2&=&-\alpha_2 y_2-y_2\sum_{i+j+k=1}^\infty b_{i,j,k}\Delta_1^i\Delta_2^j\Delta_3^k,\\
\dot y_2&=&\alpha_2 x_2 +x_2\sum_{i+j+k=1}^\infty b_{i,j,k}\Delta_1^i\Delta_2^j\Delta_3^k,\\
\dot x_3&=&-\alpha_3 y_3-y_3\sum_{i+j+k=1}^\infty c_{i,j,k}\Delta_1^i\Delta_2^j\Delta_3^k,\\
\dot y_3&=&\alpha_3 x_3 +x_3\sum_{i+j+k=1}^\infty c_{i,j,k}\Delta_1^i\Delta_2^j\Delta_3^k,
\end{array}
\]
where $a_{i,j,k},b_{i,j,k},c_{i,j,k}$ for $i+j+k=K$ depends on the coefficients of $j^K X$.

By Theorem \ref{tpd}, $X$ is formally conjugate to $\tilde X$. Now, using changes of coordinates and time-reparametrizations, we prove that $\tilde X$ is formally orbitally equivalent to a vector field of the form
\[\tilde{\tilde X}: 
\begin{array}{lcl}
\dot x_1&=&-\alpha_1 y_1-y_1 \left(\varepsilon_1\Delta_2+\varepsilon_2\Delta_3+\sum_{i=1}^\infty a_{i}\Delta_1^i\right), \\
\dot y_1&=&\alpha_1 x_1 +x_1\left(\varepsilon_1\Delta_2+\varepsilon_2\Delta_3+\sum_{i=1}^\infty a_{i}\Delta_1^i\right),\\
\dot x_2&=&-\alpha_2 y_2-y_2\left(\varepsilon_3\Delta_1+\varepsilon_4\Delta_3+\sum_{i=1}^\infty b_{i}\Delta_2^i\right),\\
\dot y_2&=&\alpha_2 x_2 +x_2\left(\varepsilon_3\Delta_1+\varepsilon_4\Delta_3+\sum_{i=1}^\infty b_{i}\Delta_2^i\right),\\
\dot x_3&=&-\alpha_3 y_3-y_3\left(\varepsilon_5\Delta_1+\varepsilon_6\Delta_2+\sum_{i=1}^\infty c_{i}\Delta_3^i\right),\\
\dot y_3&=&\alpha_3 x_3 +x_3\left(\varepsilon_5\Delta_1+\varepsilon_6\Delta_2+\sum_{i=1}^\infty c_{i}\Delta_3^i\right),
\end{array}
\]
where the coefficients are not necessarily the same as in $\tilde X$, and satisfy an additional relation, $\tilde{\tilde X}-j^3\tilde{\tilde X}=0$ is a decoupled Hamiltonian vector field.

The proof of the existence of the changes of coordinates and time-reparametrizations between $\tilde X$ and $\tilde{\tilde X}$ is an adaptation of the proof of Theorem A, and will be given in the next section.

\section{Proof of Theorem A}

We prove Theorem A and its corollaries just for the case $X\in\Omega_{0,e,\infty}^{2n}(\varphi)$. The other case can be proved using the changes indicated in Section \ref{sec.detail}.

We have to solve the system \eqref{SYSTEM}, that is,
$f_{2k+3}-g_{2k+3}+A{\bf x}\cdot \theta_{2k+2}=2f_3^\Sigma$.

It is easy to write out all the terms in this sum, except perhaps the right-hand side, $f_3^\Sigma=f_3(\sigma_1 \Delta_1,\ldots,\sigma_n\Delta_n)$.

System \eqref{SYSTEM} is equivalent to a linear system in the coefficients of $g_{2k+3}$, $\theta_{2k+2}$ and $\Sigma=(\sigma_1,\ldots,\sigma_n)$.

We write\\

\noindent $f_{2j-1}^{[k]}=-y_j\sum_{|I|=k+1}a_{j,I}^{[k]}\Delta_1^{i_1}\cdot\ldots\cdot\Delta_n^{i_n}   $, $j=1,\ldots,n$,\\
\noindent $f_{2j}^{[k]}=x_j\sum_{|I|=k+1}a_{j,I}^{[k]}\Delta_1^{i_1}\cdot\ldots\cdot\Delta_n^{i_n}   $, $j=1,\ldots,n$,

\noindent $\theta_{2k+2}({\bf x})=1+\sum_{|I|=k+1}\theta_I \Delta_1^{i_1}\cdot\ldots\cdot\Delta_n^{i_n}$,\\
\noindent $g_{2k+3}=J\nabla H_{2k+4}$, where $H_{2k+4}({\bf x})=\sum_{|I|=k+2} h_I \Delta_1^{i_1}\cdot\ldots\cdot\Delta_n^{i_n}$ and\\
\noindent $\sigma_j^{[k]}({\bf x})=\sum_{|I|=k} \mu_{j,I}\Delta_1^{i_1}\cdot\ldots\cdot\Delta_n^{i_n}$.

So \eqref{SYSTEM} is equivalent to the following linear system.
\bq\label{l.s}
\begin{array}{lcl}
a_{j,I}&=&-\alpha_j \theta_I \ (\textrm{from} \ A{\bf x}\cdot \theta_{2k+2})\\
&+&2(i_j+1)h_{i_1,\ldots,i_{j-1},i_j+1,i_{j+1},\ldots,i_n} \ (\textrm{from} \ g_{2k+3})\\
&+& \sum_{l=1}^n a_{j,e_r}^{[1]} \cdot \mu_{l,i_1,\ldots,i_{j-1},i_j-1,i_{j+1},\ldots,i_n} \ (\textrm{from} \ f_3^\Sigma),
\end{array}
\eq
for $j=1,\ldots, n$, where $e_r=(0,\ldots,0,1,0,\ldots,0)$ has $1$ in the position $r$ and $0$ otherwise. In the last summation we consider $\mu_{j,i_1,\ldots,i_{j-1},i_j-1,i_{j+1},\ldots,i_n}=0$ if $i_j-1<0$, and $a_{j,e_r}^{[1]}$ are the coefficients of $f_3$. Note that the relation $a_{j,e_r}^{[1]}=a_{r,e_j}^{[1]}$  holds between these coefficients, as $j^3 \tilde X$ is Hamiltonian.

Now put $\mathcal F=\prod_{j=1}^n\prod_{r=1}^n a_{j,e_r}^{[1]}$ and suppose $\mathcal F\neq 0$. This is a kind of nonvanishing condition for the diagonal elements of system \eqref{l.s}.

The triangular structure of \eqref{l.s}, and the condition $\mathcal F\neq 0$, guarantees that such system has a solution.

This concludes the proof of Theorem A.

\begin{proof}[Proof of Corollary 1]
The proof of Corollary 1 is a straightforward adaptation of the proof of Theorem A. The main difference is that in equation \eqref{SYSTEM} we have to consider $h_{j,I}=0$ for $I\neq (k+2)e_r$, for $j,r=1,\ldots,6$. As the condition $\mathcal F\neq 0$ is not affected by these conditions, the new system is also solvable.
\end{proof}

\begin{proof}[Proof of Corollary 2]
To prove the Corollary 2 it is suffice to adapt the proof of Theorem A. In equation \eqref{SYSTEM} we have to ignore the terms $A{\bf x}\cdot \theta_{2k+2}$, as these terms come from the time-reparametrizations. In dimension 6 there is sufficient freedom in the variables to solve the system.
\end{proof}

\section{Proof of Theorem B}
Let $X\in\Omega_{0,e}^{4}(\mathbb D_4)$ be as in the statement of Theorem B, that is, the origin is $r_1:r_2$-resonant, such that $r_1,r_2$ odd numbers with $r_1r_2>1$. 

Using the result from \cite{mt-aabc}, we obtain that a normal form of $X$ have the same expression of a nonresonant normal form. Then, the same proof of Theorem A can be applied to prove Theorem B.


\begin{thebibliography}{99}


\bibitem{mar1} {\sc R. Abraham and J. Marsden}, {\it Foundations of Mechanics}. Benjamin Cummings, London, 1978.

\bibitem{paty} {\sc F. Antoneli, P. H. Baptistelli, A. P. Dias, M. Manoel}, {\it Invariant theory and reversible-equivariant vector fields}. J. Pure Appl. Algebra {\bf 213}(5) (2009), 649-?663.

\bibitem{ap} {\sc V. I. Arnold}, {\it Arnold's problems}. Springer-Verlag, Berlin, 2004.

\bibitem{pd1} {\sc V. I. Arnold}, {\it Geometrical Methods in the Theory of Ordinary Differential Equations}.  A Series of Comprehensive Studies in Mathematics 250, Springer, 1996.
    
\bibitem{birkhoff1}{\sc G. D. Birkhoff,} {\it Dynamical Systems with Two Degrees of Freedom}. Trans. Amer. Math. Soc. {\bf 18}(2) (1917), 199--300.

 \bibitem{devaney} {\sc R. L. Devaney}, {\it Reversible diffeomorphisms and flows}. Trans. Amer. Math. Soc. {\bf 218} (1976), 89--113.


\bibitem{gaeta} {\sc G. Gaeta}, {\it Normal Forms of Reversible Dynamical Systems}. International Journal of Theoretical Physics {\bf 33}(9), 1994. 


\bibitem{mjt} {\sc A. Jacquemard, M. F. S. Lima and M. A. Teixeira}, {\it Degenerate resonances and branching of periodic orbits}. Annali di Matematica Pura ed Applicata {\bf 187}(1) (2008), 105--117.

\bibitem{rp} {\sc J. S. W. Lamb, M. F. S. Lima, R. M. Martins, M. A. Teixeira and J. Yang}, {\it On the Hamiltonian Structure of Normal Forms at Elliptic Equilibria of Reversible Vector Fields in $R^4$}. IMECC/Unicamp Research Report 05/10, 2010. Online at \url{http://www1.ime.unicamp.br/rel_pesq/relatorio.html}.

\bibitem{mt-cpaa} {\sc R. M. Martins, M. A. Teixeira}, {\it On the Similarity of Hamiltonian and Reversible Vector Fields in 4D.} Accepted, Communications in Pure and Applied Analysis, 2011.

\bibitem{mt-aabc} {\sc R. M. Martins, M. A. Teixeira}, {\it Reversible-equivariant systems and matricial equations.} Accepted, Anais da Academia Brasileira de Ci\^encias, 2011. Online at \url{http://arxiv.org/abs/0809.0299}.

\bibitem{fnbirkhoff} {\sc J. C. van der Meer}, {\it The Hamiltonian Hopf bifurcation}. Lecture Notes in Mathematics {\bf 1160}, Springer, Berlin, 1982.

\bibitem{sand} {\sc J. C. van der Meer, J. A. Sanders and A. Vanderbauwhede}, {\it Hamiltonian Structure of the Reversible Nonsemisimple 1:1 Resonance}, Dynamics, Bifurcation and Symmetry: New Trends and New Tools, Kluwer Academic Publishers, 1994.

\bibitem{baley1} {\sc G. B. Price}, {\it On Reversible Dynamical Systems}. Trans. Amer. Math. Soc. {\bf 37}(1), (1935), 51--79.

\bibitem{sev} {\sc M. B. Sevryuk}, {\it The finite-dimensional reversible KAM theory}. Phys. D {\bf 112}, (1935), 132--147.

\bibitem{marco1} {\sc M. A. Teixeira}, {\it Singularities of reversible vector fields}. Phys. D {\bf 100} (1997)(1--2), 101--118.


\end{thebibliography}
\end{document}